\documentclass[letterpaper, 10 pt, conference]{ieeeconf}
\IEEEoverridecommandlockouts           
\usepackage{times}
\usepackage[utf8]{inputenc}
\usepackage{color}
\usepackage{amsmath}
\usepackage{times}
\usepackage{graphicx}
\usepackage{color}
\usepackage{multirow}
\usepackage{blindtext}
\usepackage{amsfonts}
\usepackage{amsmath}
\usepackage{times}
\usepackage{fancyhdr}
\usepackage{amssymb}
\usepackage{verbatim}
\usepackage{xspace}
\usepackage{cite}
\usepackage{comment}
\usepackage{lipsum}
\usepackage{bm}
\usepackage{hyperref}
\usepackage{subfigure}

\usepackage[normalem]{ulem}

\usepackage{algorithm}
\usepackage[]{algpseudocode}

\usepackage{amsthm}

\newcommand{\rd}{{\mathrm d}}

\newcommand{\rp}{{\mathrm p}}
\newcommand{\re}{{\mathrm e}}

\newcommand{\rtr}{{\mathrm{tr}}}

\newcommand{\rT}{{\mathrm{T}}}

\newcommand{\fracpartial}[2]{\frac{\partial #1}{\partial  #2}}

\newcommand{\calH}{{\cal H}}
\newcommand{\calF}{{\cal F}}

\newcommand{\calL}{{\cal L}}

\newcommand{\calU}{{\cal U}}

\newcommand{\Pb}{\mathbb{P}}

\newcommand{\Rb}{\mathbb{R}}
\newcommand{\Eb}{\mathbb{E}}

\newtheorem{theorem}{Theorem}[section]

\newtheorem{lemma}[theorem]{Lemma}

\newtheorem{remark}{Remark}
\newtheorem{definition}{Definition}[section]

\begin{document}	
	\title{Stabilizing Optimal Density Control of Nonlinear Agents with Multiplicative Noise}
	\author{Kaivalya Bakshi, Evangelos A. Theodorou and Piyush Grover
	\thanks{Kaivalya Bakshi is with the Powertrain Control Division, Southwest Research Institute, San Antonio, TX 78238 (email: kaivalya.bakshi@swri.org)} 
	\thanks{Evangelos A. Theodorou is with the Department of Aerospace Engineering, Georgia Institute of Technology, Atlanta, GA 30332 (email: evangelos.theodorou@gatech.edu)}
	\thanks{Piyush Grover is with the Department of Mechanical and Materials Engineering, University of Nebraska-Lincoln, Lincoln, NE 68588}
			}
	
	\date{\today}
	
	\maketitle

\begin{abstract}
Control of continuous time dynamics with multiplicative noise is a classic topic in stochastic optimal control. This work addresses the problem of designing infinite horizon optimal controls with stability guarantees for \textit{a single agent or large population systems} of identical, non-cooperative and non-networked agents, with multi-dimensional and nonlinear stochastic dynamics excited by multiplicative noise. For agent dynamics belonging to the class of reversible diffusion processes, we provide constraints on the state and control cost functions which guarantee stability of the closed-loop system under the action of the individual optimal controls. A condition relating the state-dependent control cost and volatility is introduced to prove the stability of the equilibrium density. This condition is a special case of the constraint required to use the path integral Feynman-Kac formula for computing the control. We investigate the connection between the stabilizing optimal control and the path integral formalism, leading us to a control law formulation expressed exclusively in terms of the desired equilibrium density.
\end{abstract}

\section{Introduction}
\label{sec:intro}

Stochastic systems with multiplicative noise are used in modeling biological motion (\cite{Wolpert_2000}, \cite{Todorov2002}), signal processing (\cite{gershon2001}, \cite{Balakrishnan2002}) (where multiplicative noise is called speckle noise), atmospheric circulation \cite{Sardeshmukh2005} and in optimal control of sensori-motor systems (\cite{Todorov:2004}, \cite{Todorov2005c}). The control and estimation of such systems is therefore a classic topic with previous efforts focused on stability \cite{Haussmann1972} and estimation \cite{Phillis_1989} of linear systems, Lyapunov based stability of nonlinear \cite{Krstic2001} systems, robustness of discrete \cite{gershon2001} systems, and more recently in an optimal control setting \cite{Todorov_MatrixInequalities}. In the context of control, stability analysis and optimization are both challenging aspects. The multiplicative noise severely affects stability and robustness \cite{Ding2016} and the optimal control formulation in this case is non-standard due to the \textit{non-Gaussianity} of the dynamics \cite{Phillis_1989}, and most works addressing either aspect in the linear-quadratic regime. This is also true in the case  of non-Gaussianity due to jump noise (\cite{Bakshi_CDC2016}, \cite{Pham2014}). In the case of nonlinear dynamics with multiplicative noise, differential dynamic programming \cite{Todorov_MatrixInequalities} and forward backward stochastic differential equations using first \cite{Yannis2016} and second order schemes (\cite{Bakshi_ACC2017}, \cite{Theodorou2019RSS}) have been used to synthesize algorithms to compute the control.

The dynamics and control of large, multi-agent populations consisting of identical and non-cooperative agents are of interest in various applications including robotic swarms, macro-economics, traffic and neuroscience. The fundamental idea of using continuum models for distributed decision making with limited information in large populations, yields a framework for regulation using local feedback. Optimal feedback control formulations to model and control such  \textit{large-size} populations of \textit{small} self-propelled agents with individual state-feedback controllers \cite{LiBerman2017}, utilize a tractable representation consisting of coupled Hamilton-Jacobi-Bellman (HJB) and Fokker-Planck (FP) partial differential equations (PDEs) governing the control law and density evolution respectively. Lately, efforts on this topic have been focused on the case of \textit{networked} agents, with the agents linked through the dependence of an individual's dynamics or utility on the state of other agents. On the topic of \textit{mean-field games} (MFGs) which are used to describe \textit{emergent} behavior, arising from the networked case, the works (\cite{GueantLions2011}, \cite{yinmehmeysha12}) proposed a framework for analyzing linear stability of the density of agents with simple integrator dynamics. Linear stability and bifurcations of a flocking model \cite{grover2018mean}, \textit{control deign} of stabilizing infinite-horizon controllers for general MFGs \cite{Bakshi2018TCNS} as well as nonlinear stability of large population optimal control systems for non-networked agents \cite{Bakshi2020TAC} with Langevin dynamics have also been analyzed. In contrast to the general MFG case, in the non-networked case, individual agent dynamics are not affected by other agents in the population via explicit dependence of the individual dynamics or the cost function on the states of those agents.

While optimal density control models \cite{Lasry2007} have been formulated for nonlinear agents with multiplicative noise dynamics (section \ref{sec:formulation}), prior works do not address the design of stabilizing infinite-horizon controls for such systems. In this work, we introduce a framework for analyzing the nonlinear stability of the PDE optimality system governing \textit{optimal control models for a single agent or large populations}. We specify design constraints (section \ref{sec:design}) required to guarantee closed-loop stability of the desired stationary density of the system under the action of the individual steady state controls, when the agent dynamics belong to the class of reversible diffusions. Finally, we discuss the connection (section \ref{sec:path_integral_connection})  between the introduced design constraints and the path integral formulation used in the solution of the HJB equation governing the control law.

\section{Optimal control}
\label{sec:formulation}

We provide general notation and then describe the control problem considered in this work. We denote vector inner products by $a \cdot b$, induced Euclidean norm by $|a|$, it's square by $a^2 = |a|^2$, square weighted norm by a matrix $M$ of suitable dimensions by $|a|^2_M = a^\rT M a$, partial derivative with respect to $t$ by $\partial_t$, while $\nabla$, $\nabla \cdot$, $\Delta$ and $\nabla \cdot M (x)$ denote the gradient, divergence, Laplacian and matrix gradient w.r.t. $x \in \Rb^n$ of $M:\Rb^n \rightarrow \Rb^{n \times n}$  operations respectively, with $(\nabla \cdot M)_i := \sum_{k = 1}^n \fracpartial{M_{ik}}{x_k}$. $L^2(\Rb^n)$ denotes the class of square integrable functions of $\Rb^n$ and $L^2(\Rb^n; f(x) \rd x)$ denotes the class of $f$ weighted square integrable functions. The norm of a function $f$ and inner product of functions $f_1, f_2$ in this class is denoted by $||f||_{L^2(\Rb^n)}$ and $\big< f_1,f_2 \big>_{L^2(\Rb^n)}$ respectively.

\subsection{Formulation}
\label{sec:LargeScale_OCP}

Let $x_s$, $u(s)$ $\in \Rb^n$ denote the state and control inputs of the first-order, controlled, stochastic system dynamics of a representative agent or a solitary agent,
\begin{equation}
	\rd x_s = b(x_s) \rd s + u(s) \rd s + \sigma(x_s) \rd w_s \label{dyn}
\end{equation}
for every $s \geq 0$, excited by multiplicative noise driven by a standard $\Rb^m$ Brownian motion, with volatility $0 < \sigma: \Rb^n \rightarrow \Rb^{n \times m}$, on the filtered probability space $\{ \Omega, \calF, \{\calF_t\}_{t \geq 0},\Pb \}$. In this work we assume that the drift has the form $b(x) = \nabla \cdot \Sigma(x)/2 - \Sigma(x)  \nabla \phi(x)/2$ where we define the diffusion matrix by $\Sigma := \sigma \sigma^\rT$ and assume that $\Sigma$ is uniformly positive definite and that the functions $\sigma$ and $\phi$ are smooth. Due to these assumptions, the dynamics \eqref{dyn} are the controlled version of reversible diffusion dynamics characterized in definition 4.3, p 116 in \cite{Pavliotis2014}. Note also that the passive dynamics are a generalized version of Langevin dynamics \cite{Bakshi2018TCNS} with multiplicative noise in the overdamped case. It is explained in section \ref{sec:design} why dynamics of this class are reversible diffusions, and how our assumption of this structure simplifies control design and facilitates the subsequent stability analysis. The control $u \in \calU := \calU[0,T]$ where $\calU$ is the class of admissible controls \cite{YongBook_1958} containing functions $u:[0,T] \times \Rb^n \rightarrow \Rb^n$. Consider the following long time average utility optimal control problem (OCP)
\begin{equation}
\underset{u \in \calU}{\min} J(u) := \underset{T \rightarrow +\infty}{\lim} \frac{1}{T} \Eb \left[ \int_{0}^{T} q(x_s) \rd s + \frac{1}{2}|u(s)|^2_R \; \rd s \right] \label{OCP}
\end{equation}
subject to \eqref{dyn} which we refer to as problem \textbf{(P)}, wherein the expectation is calculated on the probability density $p(s,x)$ of the state $x_s$ for all $0 \leq s$ which represents the distribution of the population of agents, with the initial density being $x_0 \sim p(0,x)$, $q: \Rb^n \rightarrow \Rb$ is a known deterministic function which has at-most quadratic growth and is bounded from below and $0 < R: \Rb^n \rightarrow \Rb^{n \times n}$ is the state-dependent control cost. We assume that $b(\cdot)$, $\Sigma(\cdot)$, $q(\cdot)$, $R(\cdot)$ and functions in the class $\calU$ are measurable. \textit{Indeed, the problem \textbf{(P)} and the presented results apply to the density control of a single agent system as well,} although we express results for the large population systems in the following.

\subsection{PDE representation}

The standard application of \textit{dynamic programming} \cite{Fleming2006} implies that under certain regularity conditions \cite{yinmehmeysha12} which we assume to be true, the problem \textbf{(P)} is equivalent to the HJB-FP PDE optimality system governing the value and density functions respectively, given by
\begin{align}
q - c - \frac{1}{2}|\nabla v^\infty|^2_{R^{-1}} + \nabla v^\infty \cdot b + \frac{1}{2} \rtr (\Sigma \Delta v^\infty) =& 0 \label{s1} \\
- \nabla \cdot ((b - R^{-1} \nabla v^\infty) p^\infty) + \frac{1}{2} \Delta (\Sigma p^\infty) =& 0 \label{s2}
\end{align}
with the constraint $\int p^\infty \rd x = 1$ where $c$ is the optimal cost. The optimal control is given by $u^\infty(x) = - R^{-1} \nabla v^\infty$. Under certain regularity conditions  \cite{yinmehmeysha12} which we assume to be true,
the time-varying relative value \cite{yinmehmeysha12} function and density corresponding to problem \textbf{(P)} are governed by the optimality system
\begin{align}
-\partial_t v =& q - c - \frac{1}{2}|u^*|^2_{R} - R u^* \cdot b - \frac{1}{2} \rtr( \Sigma \nabla (R u^*)) \label{HJB} \\
\partial_t p =& - \nabla \cdot ((b + u^*)p) + \frac{1}{2} \Delta (\Sigma p) \label{FP}
\end{align}
with the constraint $\int p(t,x) \rd x = 1$ for all $t \geq 0$. The optimal control is given by $u^*(t, x) = - R^{-1} \nabla v$. In this work, we assume the additional conditions \cite{YongBook_1958} which are required to show that the HJB PDEs \eqref{s1} and \eqref{HJB} have unique solutions. Note that both the steady state and time-varying HJB PDEs are semi-linear. We denote the the generator of the optimally controlled stationary process \eqref{dyn} as $\calL(\cdot) := (b - R^{-1} \nabla v^\infty) \cdot \nabla (\cdot) + (1/2) \rtr(\Sigma \Delta (\cdot))$ and its $L^2(\Rb^n)$ adjoint as $\calL^\dagger (\cdot) := - \nabla \cdot ((b - R^{-1} \nabla v^\infty) (\cdot)) + (1/2) \Delta (\Sigma (\cdot)) = - \nabla \cdot J(\cdot)$. 
The control synthesis proposed in \ref{sec:design} enables us to obtain an analytic form of the stationary density in terms of the value function. However, analytical solution of HJB equations arising from stochastic control problems is quite challenging, most examples of which are restricted to the linear-quadratic regime. Therefore, in this work we obtain sufficiency conditions required to show stability of the (analytically) unknown stationary density by assuming the existence of the solution to the stationary optimality system.
\begin{itemize}
    \item[\textbf{(A0)}] There exist $(v^\infty(x), p^\infty(x)) \in (C^2(\Rb^n))^2$ satisfying the optimality system (\ref{s1}, \ref{s2}).
\end{itemize}
\begin{remark}\label{finitetime}
	The finite time OCP analogous to the infinite time OCP \textbf{(P)} given by
	\begin{equation}
	\underset{u \in \calU}{\min} J(u) := \Eb \left[ \int_{0}^{T} q(x_s) \rd s + \frac{1}{2} |u(s)|^2_{R^{}} \; \rd s \right], \label{finitetimeOCP}
	\end{equation}
	subject to the dynamics \eqref{dyn} has the optimality system given by equations (\ref{HJB}, \ref{FP}) with $c = 0$, initial density given by $p(0,x)$ and constraint $\int p(t,x) \rd x = 1$.
\end{remark}
\begin{remark}\label{dicountcost}
	For the OCP with discounted cost utility
    \begin{equation}
        \underset{u \in \calU}{\min} J(u) := \Eb \left[ \int_{0}^{T} \re^{-\rho s} \left( q(x_s) \rd s + \frac{1}{2} |u(s)|^2_{R^{}} \; \rd s \right) \right], \label{discountOCP}
    \end{equation}
	subject to the dynamics \eqref{dyn}, the optimality system is given by equations (\ref{HJB}, \ref{FP}) with $c$ replaced by $\rho v$, initial density given by $p(0,x)$ and constraint $\int p(t,x) \rd x = 1$. Similarly for the infinite horizon case, one obtains the HJB-FP system analogous to equations (\ref{s1}, \ref{s2}), on replacing c by $\rho v$. Hence, the results of this paper can be easily extended to this case.
\end{remark}

\section{Stabilizing design}
\label{sec:design}

Given the state and control cost functions $q(\cdot)$ and $R(\cdot)$, which obey the regularity conditions mentioned in the section \ref{sec:formulation}, we can solve the system (\eqref{s1}, \eqref{s2}) to obtain the optimal control and the corresponding steady state control $u^\infty(\cdot)$ and the optimally controlled density $p^\infty(\cdot)$. If the designed cost functions result in a stabilizing steady state control (conditions for which we investigate in this work), a perturbation in the state represented by the density perturbation $\tilde{p}(s,\cdot)$ should  vanish under the action of the corresponding steady state optimal control $u^\infty$ obtained by solving equation \eqref{s1}, after sufficient passage of time. Evolution of the corresponding perturbed density $p(s,\cdot) = p^\infty(\cdot)(1 + \tilde{p}(s, \cdot))$ is governed by the equation \eqref{FP} with $u^*$ replaced by $u^\infty$. The stability analysis then corresponds to the time-varying FP equation
\begin{equation}
    \partial_t p = - \nabla \cdot ( (b - R^{-1} \nabla v^\infty) p) + \frac{1}{2} \Delta (\Sigma p) \label{FPnew}
\end{equation}
with the constraint $\int p(t,x) \rd x = 1$ for all $t \geq 0$. The reason for this choice of the structure of the perturbation will become clear in the analysis presented in this section. 

The design question we investigate in this work is: given the control formulation \textbf{(P)}, what are the constraints on the cost functions $q(\cdot)$ and $R(\cdot)$, which if satisfied, guarantee that the perturbed density $p^\infty(\cdot) + \tilde{p}(s,\cdot)$ decays to the equilibrium density $p^\infty(\cdot)$, given sufficient time? A stabilizing optimal control for the system \eqref{dyn} can then be synthesized, using the a priori knowledge of these constraints.

\subsection{Reversible diffusions}

The first challenge in showing stability of equilibrium distributions of nonlinear stochastic systems with PDE representations, is the analytical expression of the stationary density. Obtaining analytical solutions to the stationary FP equation is a non-trivial task, especially in the presence of multiplicative noise, and most solutions in the literature appear within the context of linear dynamics \cite{BardiBook}. However, it is possible to obtain Gibbs-like formulae, under certain conditions. For the system considered \eqref{dyn}, we first provide some background on the reversible diffusions, in the uncontrolled case. Then, we explain their usefulness in stability analysis and outline a control synthesis strategy which imposes the reversible diffusions structure on the controlled dynamics.

Setting the control to zero in equation \eqref{dyn} recovers the (uncontrolled) reversible diffusion dynamics. In this case, the stationary distribution satisfying the FP equation \eqref{s2} can be written, according to equation (4.96), p 118 in \cite{Pavliotis2014}, as
\begin{equation}
p^\infty = \frac{\exp(-\phi)}{Z} \label{stationarySolRevDiff}
\end{equation}
under certain conditions, where $Z = \int \exp(-\phi (x)) \rd x$, since the \textit{generalized potential} \cite{Pavliotis2014} $\phi$ satisfies the \textit{detailed balance} \cite{Cerfon2015} condition $b = \nabla \cdot \Sigma(x)/2 - \Sigma(x)  \nabla \phi(x)/2$. This condition for \eqref{stationarySolRevDiff} to be the stationary density for the dynamics \eqref{dyn} with the general drift $b(\cdot)$ and control equal to zero, is obtained directly by substituting the density \eqref{stationarySolRevDiff} in equation \eqref{s2}. This structure of the drift allows us to obtain the analytic form of the stationary density and the consequent useful eigen properties of the generator and its adjoint, which are critical to the stability analysis of the FP PDE, the state transition of which depends on the generator. Consequently, we propose a control synthesis strategy in which we force the optimally controlled dynamics to be reversible diffusions.

\subsection{Stationary solution}

We introduce a constraint relating the state-dependent control cost and the volatility, which causes the optimally controlled dynamics to belong to the class of reversible diffusions. This allows us to obtain an analytical form of the stationary density in terms of the steady state value function besides giving us useful eigen properties of the generator of the process, which are required in the stability analysis.
\begin{itemize}
    \item[\textbf{(A1)}] $R^{-1}(x) = \Sigma(x)/2$ for all $x \in \Rb^n$.
\end{itemize}
\begin{remark} \label{pathIntegral}
    The constraint of type \textbf{(A1)} is used in the context of path integral sampling based solutions to the HJB equation \eqref{HJB} using the Feynman-Kac lemma \cite{Kappen2005a}.
\end{remark}
\begin{lemma} \label{LemmastationaryRevDiff}
	Let \textbf{(A0)} and \textbf{(A1)} be true. If the process $x_s$ which obeys \eqref{dyn} under the steady state optimal control $u^\infty$ obtained by solving equation \eqref{HJB}, is stationary, then it is a reversible process with the stationary density
    \begin{equation}
        p^\infty = \frac{\exp(-\Phi)}{Z} \label{stationarySolRevDiffCtrl}
    \end{equation}
	satisfying the FP equation \eqref{s2}, where $\Phi := \phi + v^\infty$ and $Z = \int \exp(-\Phi (x)) \rd x$.
\end{lemma}
\begin{proof}
    We observe that under the assumption \textbf{(A1)}, the dynamics \eqref{dyn} under $u^\infty$ obeys the dynamics \eqref{dyn} where $b$ satisfies the \textit{detailed balance} condition
    \begin{align}
        b =& \nabla \cdot \Sigma/2 - \Sigma \nabla (\phi + v^\infty)/2. \label{detailedBalance_ssCtrl}
    \end{align}
    The proof then follows from proposition 4.5, p 119 in \cite{Pavliotis2014}. Further, from the condition \eqref{detailedBalance_ssCtrl} we observe that the steady state density satisfying the stationary FP equation \eqref{s2} is given by \eqref{stationarySolRevDiffCtrl}, where $\Phi = \phi + v^\infty$ is the \textit{optimal generalized potential} corresponding to the OCP \textbf{(P)}.
\end{proof}

\subsection{Eigen properties}

\begin{lemma} \label{LemmaRevDiffSelfAdjGen}
	{Let \textbf{(A0)} and \textbf{(A1)} be true. The stationary diffusion with dynamics \eqref{dyn} under the steady state optimal control $u^\infty$ obtained by solving \eqref{s1}, generator $\calL$ and stationary density $p^\infty$ satisfying \eqref{s2}, is reversible, if and only if its generator is self-adjoint in $L^2(p^\infty(x) \rd x; \Rb^n)$.}
\end{lemma}
\begin{proof}
    The proof follows by application of theorem 4.5, p 116 in \cite{Pavliotis2014} to the case of the diffusion dynamics with drift given by equation \eqref{detailedBalance_ssCtrl}.
\end{proof}
We introduce an assumption on the optimal generalized potential to obtain relevant properties of the generator of the controlled process. This is an implicit condition on the cost function $q$ and specifies its dependence on the value, diffusion and the generalized potential  functions.
\begin{itemize}
    \item[\textbf{(A2)}] $\underset{|x| \rightarrow +\infty}{\lim} \big( \frac{|\nabla \Phi|^2}{2} - \Delta \Phi \big) = +\infty$.
\end{itemize}
\begin{remark}
    If we set the control to zero in the dynamics \eqref{dyn}, the resulting diffusion may not be stationary and hence stable. If it is (conditions given in \cite{Pavliotis2014}), the unique invariant density is given by \eqref{stationarySolRevDiff}, which in general, will not be the desired stationary distribution. Our objective is to utilize the optimization in section \ref{sec:formulation} to stabilize and control the density evolution of the process \eqref{dyn}. The condition \textbf{(A2)} in turn, specifies an implicit control design constraint on the cost function $q$, which is required to guarantee stability of the terminal distribution attained by the optimally controlled process \eqref{dyn} under the control law satisfying the OCP \textbf{(P)}.
\end{remark}
\begin{lemma}\label{ThmpotlnpOperSpecGap}
	Let \textbf{(A0)}, \textbf{(A1)} and \textbf{(A2)} hold. Then $p^\infty(x)$ satisfying \textbf{(A0)} and given by \eqref{s2}, satisfies the Poincar\'e inequality with $\lambda > 0$, that is, 
	there exists $\lambda > 0$ such that for all functions $f \in C^1(\Rb^{d})\cap L^2(p^\infty(x) \rd x; \Rb^{d})$ and $\int f p^\infty(x) \rd x = 0$, we have
	\begin{align}
		&\lambda \frac{2}{\sigma^2} ||f||^2_{L^2(p^\infty(x) \rd x; \Rb^{n})} \notag \\
		&\leq ||\nabla f||_{L^2(p^\infty(x) \rd x; \Rb^{n})} = -\big<\calL f, f\big>_{L^2(p^\infty(x) \rd x; \Rb^{n})}.
	\end{align}
\end{lemma}
\begin{proof}
    Due to assumption \textbf{(A2)}, we know that the $\Phi$ satisfies the Poincar\'e inequality \cite{Bakry2014}, implying by theorem 4.3, p 112 in \cite{Pavliotis2014} that it satisfies the \textit{spectral gap} property above.
\end{proof}

Lemma \ref{LemmaRevDiffSelfAdjGen} implies that eigenvalues of $\calL$ are real, negative semi-definite and its eigenfunctions are orthonormal in $L^2(\rp^\infty(x) \rd x; \Rb^{d})$ while lemma \ref{ThmpotlnpOperSpecGap} implies that the eigenvalues of $\calL$ are discrete and its eigenfunctions are complete on $L^2(\rp^\infty(x) \rd x; \Rb^{d})$ \cite{Cerfon2015}.  We denote the eigenvalues $\{\xi_n\}_{n \geq 0}$ and corresponding eigenfunctions $\{\Xi_n\}_{n \geq 0}$  of $\calL$ which form a complete orthonormal basis of $\calH$. Let eigenvalues $\{\xi_n\}_{n \geq 0}$ be indexed in descending order of magnitude $... < \xi_n < \xi_1 < ... < \xi_0 = 0$ and let $\Xi_0 = 1$.

\subsection{Perturbation equation}

The decay of an initial density under uncontrolled (open-loop) overdamped Langevin dynamics (with constant volatility) to a stationary density is a classical topic \cite{Risken1986}. We address the question of decay of a perturbed density of the system \eqref{dyn} under the closed-loop steady state optimal controls. The perturbation analysis then corresponds to the time-varying FP equation \eqref{FPnew}. The proposed approach introduces a framework for analyzing stability of densities of controlled nonlinear stochastic systems with multiplicative noise.
\begin{theorem} \label{ThmgenPertMFG}
	{Let (\textbf{A0, A1, A2}) be true. If $(v^\infty,p^\infty)$ are steady state solutions to the optimality system (\ref{s1}, \ref{s2}), then the perturbation $\tilde{p} \in C^{1,2}([0, +\infty) \times \Rb^n)$ of the equilibrium density $p^\infty$, with the corresponding time-varying density $p = p^\infty(1 + \tilde{p})$ of the state, obeys
	\begin{align}
	\partial_t \tilde{p} =& \calL \tilde{p}, \label{pert}
	\end{align}
	where $\tilde{p}(0,x)$ is given and $\int_{\Rb^n}p^\infty(x)(1 + \tilde{p}(t,x)) \rd x = 1$ for all $t \geq 0$.}
\end{theorem}
\begin{proof}
	Substituting the perturbation density {$\rp = \rp^\infty(1 + \tilde{\rp})$} in equation \eqref{FPnew} written in terms of the operator $\calL^\dagger$ and using equation \eqref{s2}, we have
	{\begin{align}
    \partial_t (\rp^\infty \tilde{\rp}) = & \calL^\dagger (\rp^\infty \tilde{\rp}).
	\end{align}}
	It can be verified \cite{Cerfon2015} that the process \eqref{dyn} satisfies the \textit{detailed balance} property (which is true due to equation \eqref{detailedBalance_ssCtrl}) if and only if the generator $\calL$ and its adjoint satisfy $\calL^\dagger(\rp^\infty f) = \rp^\infty \calL f$  for any smooth function $f$. Thus, we obtain the density perturbation equation \eqref{pert}. The mass conservation constraint follows directly from the same constraint on equation \eqref{FPnew}.
\end{proof}

\subsection{Stability analysis}

We show stability of the equilibrium density for a class of perturbations living in a Hilbert space  defined below.
\begin{definition}
Let \textbf{(A1)} hold. Denote the density $p^\infty(x)$ given in equation \eqref{stationarySolRevDiffCtrl}, where $(v^\infty, p^\infty)$ is a pair satisfying \textbf{(A0)}. Denote by $\calH$ the Hilbert space $L^2(p^\infty(x) \rd x; \Rb^{n})$. The class of mass preserving density perturbations is defined as $S_0 := \bigg{\{} q(x) {\in \calH}  \bigg| \big< 1,q(x) \big>_{\calH} = 0 \bigg{\}}$.
\end{definition}
\begin{definition} \label{def:S_lnp}
Let us denote the set of initial perturbed densities by $S = \bigg\{  p(0,x)= p^\infty(x)(1 + \tilde{\rp}(0,x))\bigg|  p(0,x) \geq 0, \tilde{p}(0,x)\in S_0\bigg\}$. We say the equilibrium stationary density $p^\infty(x)$ of the FP equation \eqref{FP} is asymptotically stable with respect to $S{}$ if there exists a solution $\tilde{p}(t,x)$ to the perturbation equation \eqref{pert} such that $\underset{t \rightarrow +\infty}{\lim}||\tilde{p}(t,x)||_{\calH}= 0$.
\end{definition}
Since we are concerned with stability of isolated stationary densities, we \emph{do not} assume that initial perturbations are mean preserving \cite{Gueant2009}. We give a proof of the stability of the equilibrium density of the closed-loop system controlled by the steady state optimal controls under perturbations belonging to the class defined above.
\begin{theorem}\label{ThmStab}
Let \textbf{(A0)} \textbf{(A1)} and \textbf{(A2)} be true. Let $\left(v^\infty(x), p^\infty(x) \right)$ be the unique stationary solution to the optimality system (\ref{s1}, \ref{s2}). If $\tilde{p}(0,x) \in S_0$ and $\{p_n\}_{0 \leq n \leq +\infty}$ are determined by 
\begin{equation}
\dot{p_n}(t) = \xi_n p_n(t), \label{pertcoeffsODE}
\end{equation}
then $\tilde{p}(t,x) = \sum_{n = 1}^{+\infty} p_n(t) \Xi_n(t)$ is the unique $\calH$ solution to the perturbation equation \eqref{pert}. $p^\infty(x)$ is asymptotically stable with respect to $S$.
\end{theorem}
\begin{proof}
Since $\tilde{p}(0,x) \in \calH$ we have the unique representation $\tilde{p}(0,x) = \sum_{n = 0}^{+\infty} p_n(0) \Xi_n(x)$ where $p_n(0) = \left< \tilde{p}(0,x), \Xi_n(x) \right>_\calH < +\infty$ for all $n$.
Since $\{\Xi_n\}_{0 \leq n < +\infty}$ is a complete basis on $\calH$, any solution in $\calH$ to the PDE \eqref{pert} must have the form $\sum_{n = 0}^{+\infty} p_n(t) \Xi_n(x)$ where $\{p_n(t)\}_{0 \leq n \leq +\infty}$ are finite for all $t \in [0,+\infty)$. Substituting the selected form of the solution in the perturbation equation \eqref{pert} and using the eigen property $\calL \Xi_n = \xi_n \Xi_n$, we obtain the ODEs \eqref{pertcoeffsODE}. Due to the assumption \textbf{(A2)}, the eigen properties of the generator given in lemmas \ref{LemmaRevDiffSelfAdjGen} and \ref{ThmpotlnpOperSpecGap} hold. Using these properties yields the ODEs \eqref{pertcoeffsODE} with the unique solutions ${p_n}(t) = p_n(0) \re^{\xi_n t}$. Therefore $\tilde{p}(t,x) = \sum_{n = 0}^{+\infty} p_n(t) \Xi_n(x)$ wherein ${p_n}(t) = p_n(0) \re^{\xi_n t}$ is the unique $\calH$ solution to the perturbation equation \eqref{pert}. Further, $\tilde{p}(0,x) \in S_0$ implies that $p_0(0) = \left< \tilde{p}(0,x), \Xi_0(x) \right>_\calH = \left< \tilde{p}(0,x), 1 \right>_\calH = 0$ implying that $p_0(t) = 0$ for all $0 \leq t$, which completes the first part of the proof. \\
\indent Using Parseval's identity we have that $||\tilde{p}(t,x)||_{L^2(\Rb^n)} = \left( \sum_{n = 0}^{+\infty} p_n(t)^2 \right)^\frac{1}{2}$. Noting that $p_0(t) = 0$ and $p_n(t)^2 = p_n(0)^2 \re^{2 \xi_n t}$, where $\xi_n < 0$ for all $1 \leq n$ and using the Lebesgue dominated convergence theorem for the limit $t \rightarrow +\infty$, we have that $p^\infty(x)$ is asymptotically stable with respect to $S$.
\end{proof}

\section{Path integral connection}
\label{sec:path_integral_connection}

In this section we explore the remark \ref{pathIntegral} and apply the path integral solution of the value function to express the control law in terms of feasible stationary densities. We rewrite the HJB-FP optimality system investigated in section \ref{sec:design} below
\begin{align}
    0 =& q - c - \frac{1}{2}|\nabla v^\infty|^2_{R^{-1}} + \nabla v^\infty \cdot b + \frac{1}{2} \rtr (\Sigma \Delta v^\infty), \notag \\
    \partial_t p =& - \nabla \cdot ( (b - R^{-1} \nabla v^\infty) p) + \frac{1}{2} \Delta (\Sigma p), \notag
\end{align}
where $p = p^\infty(1 + \tilde{p})$ is the perturbed density and the optimal control to be used in the dynamics \eqref{dyn} is given by $u^\infty = - R^{-1} \nabla v^\infty$. The results of section \ref{sec:design} imply that this steady state optimal control is stabilizing if the constraints (\textbf{A0}, \textbf{A1}, \textbf{A2}) are true. Let us assume that the constraints hold true. The steady state value function, and the corresponding feedback control, may be approximated by instead solving the transient HJB equation \eqref{HJB} in the limit $T\rightarrow +\infty$ \cite{Opper2017} via the sampling based solution given by the Feynman-Kac lemma. We rewrite the transient HJB equation on substituting the control law as
\begin{align}
    -\partial_t v =& q - c - \frac{1}{2}|\nabla v|^2_{R^{-1}} + \nabla v \cdot b + \frac{1}{2} \rtr( \Sigma \Delta v). \label{HJB_final}
\end{align}
Recalling from the Feynman-Kac formula based path integral control formulation \cite{Kappen2005b} of the HJB equation, the constraint required for linearizing the PDE above is given by $\lambda  R^{-1} = \Sigma$, where $0< \lambda$ is a constant, and the change of variable used is $\Psi = \exp(-v/\lambda)$. The variable $\Psi$ governed by the linear equation
\begin{equation}
    -\partial_t \Psi = -\frac{(q - c) \Psi}{\lambda} + \nabla \Psi \cdot b + \frac{1}{2} \rtr( \Sigma \Delta \Psi ), \label{desireability}
\end{equation}
is referred to as the desire-ability \cite{Theodorou2011IPI} and characterizes the control as $u^* = \lambda R^{-1} \nabla \Psi / \Psi = \Sigma \nabla \Psi / \Psi$. Notice that $\lambda = 2$ then yields the constraint \textbf{(A1)}. The desire-ability may then be computed using the path integral formula 
\begin{equation}
    \Psi(t,y) = \Eb \left[ \exp\left( - \int_t^T \frac{q(x^y_s) - c}{\lambda}  \rd s \right) \psi(T,x^y_T) \right], \label{desireability_expression}
\end{equation}
where $x^y_s$ is the solution to the uncontrolled process \eqref{dyn} sampled by setting the control to zero with $x_t = y$ almost surely, and $c$ is the optimal cost for the OCP \textbf{(P)}, which can be computed as the first eigenvalue of the generator $\calL$ (\cite{todorov2009eigenfunction}, \cite{Bakshi2018TCNS}, \cite{Bakshi2020TAC}). The path integral solution above can then be used to approximate the time-invariant control law $u^\infty \approx u^* = -R^{-1} \nabla v$ with the value function $v^\infty \approx v = -\lambda \ln(\Psi)$ by using a large sampling horizon $T$ in equation \eqref{desireability_expression}, as in \cite{Opper2017}.

The path integral connection to the stabilizing optimal control considered in this work, is then obtained from equation \eqref{stationarySolRevDiffCtrl} as the relationship $p^\infty = \Psi^2 \exp(-\phi)/Z$ between the desire-ability given by \eqref{desireability_expression} (assuming a large time horizon $T$) and the stable equilibrium density, where $Z = \int p^\infty(x) \rd x$. Consequently, this connection allows us to analytically express the stabilizing steady state controls in terms of the corresponding stationary density as 
\begin{equation}
    u^\infty = R^{-1} (\nabla p^\infty + p^\infty \nabla \phi) / p^\infty. \label{optCtrl_pathIntegral}
\end{equation}
Therefore, using \textbf{(A1)}, the controlled system dynamics \eqref{dyn} under the stabilizing optimal control can be written exclusively in terms of the desired equilibrium density as
\begin{align}
    \rd x_s 
    =& \frac{\Sigma(x_s)}{2} \rd s + \frac{\Sigma}{2} \frac{\nabla p^\infty(x_s)}{p^\infty(x_s)} \rd s + \sigma(x_s) \rd w_s. \label{opt_dyn}
\end{align}
\begin{remark}
    The interpretation above implies that if there exists a stationary density corresponding to the system \eqref{dyn} under a control law satisfying the OCP \textbf{(P)} and constraints \textbf{(A0, A1, A2)}, then the steady state control law is given by \eqref{optCtrl_pathIntegral} and the optimally controlled dynamics are given by \eqref{opt_dyn}.
\end{remark}

Further, given a (feasible) density $p^\infty$ for which we would like to design an optimal control formulation \textbf{(P)}, we can use the path integral connection studied here to analytically calculate the corresponding required cost function. This is the \textit{inverse optimal control problem} corresponding to the OCP \textbf{(P)}, consisting of the explicit design of the cost function to attain a desired steady state density. In \cite{Opper2017}, the inverse optimal control problem was solved numerically by approximating the steady state desire-ability using a large time horizon, as explained previously. Using a similar idea, we may now approximate $\Psi = \sqrt{Z p^\infty \exp(\phi)}$ and apply this relationship to the equation \eqref{desireability} with $\partial_t \Psi$ set to zero (since we choose a large T), in order to obtain the required cost function $q$ analytically, for the case of controlled reversible diffusions. This will result in the analytical solution to the inverse OCP corresponding to \textbf{(P)} in the case of reversible diffusion dynamics, which are a special case of the more general nonlinear stochastic processes considered in \cite{Opper2017}.

\section{Conclusions}
\label{sec:conclusions}

This work addressed the design of stabilizing optimal control for a single agent or large population of identical, non-cooperative, non-networked agents obeying nonlinear dynamics with multiplicative noise. The first contribution is a framework for analyzing the nonlinear stability of the FP equation governing the density evolution of a continuum of such agents. The reversible diffusion agent dynamics are chosen for analysis, since this allows an (\textit{non-Gaussian}) analytical form of the equilibrium Gibbs density, thus facilitating stability analysis. The second contribution is providing design constraints on the state and control cost functions in the case of control affine dynamics with quadratic (state-dependent) control cost. The stability results apply to general nonlinear density perturbations belonging to a Hilbert space. The final contribution is to apply the connection of the noise-intensity | control-authority relationship with the path integral solution to obtain a control law expression exclusively in terms of the desired stationary density for the distribution of agents. This (density based feedback) form of the control law allows a heuristic analytical solution to the inverse OCP.

The proposed framework generalizes the work  \cite{Bakshi2020TAC} concerning stability of non-networked large size populations of optimally controlled agents excited by constant volatility white noise. However, the approach presented here is distinct from the variable transformation approach in \cite{Bakshi2020TAC} and analyzes stability of the non-transformed FP equation. As such, the techniques presented here can be (in contrast to \cite{Bakshi2020TAC}) directly applied to the case of networked continuum systems, that is the MFGs, with agents excited by multiplicative noise. In general, there are multiple equilibria corresponding to the two-way coupled HJB-FP optimality systems of such MFGs, especially in the case of agents with nonlinear dynamics. Further, the stability analysis in that case corresponds to that of both the HJB \textit{and} FP equations as opposed to just the FP equation, as in this paper. Investigation of the control design constraints for specific MFG formulations (\cite{GueantLions2011}, \cite{Bakshi2018TCNS}) remains to be explored. Extension of this framework to the case of second order dynamics \cite{Halder2019} and multiplicative noise also needs to be addressed.

\bibliographystyle{unsrt}
\bibliography{References}

\end{document}